\documentclass[11pt]{amsart}

\usepackage[a4paper,hmargin=3.5cm,vmargin=4cm]{geometry}
\usepackage{amsfonts,amssymb,amscd,amstext}
\usepackage{graphicx}
\usepackage[dvips]{epsfig} 
\usepackage{quoting}
\usepackage{hyperref} 

\usepackage{fancyhdr}
\input xy
\xyoption{all}

\usepackage{enumerate}
\usepackage{titlesec}
\usepackage{mathrsfs}

\titleformat{\section}
{\filcenter\bfseries} {\thesection{.}}{0.2cm}{}
\titleformat{\subsection}[runin]
{\bfseries} {\thesubsection{.}}{0.15cm}{}[.]
\titleformat{\subsubsection}[runin]
{\em}{\thesubsubsection{.}}{0.15cm}{}[.]

\usepackage[up,bf]{caption}

\newtheorem{theorem}{Theorem}[section]
\newtheorem{proposition}[theorem]{Proposition}
\newtheorem{lemma}[theorem]{Lemma}
\newtheorem{claim}[theorem]{Claim}
\newtheorem{corollary}[theorem]{Corollary}

\theoremstyle{definition}
\newtheorem{definition}[theorem]{Definition}
\newtheorem{remark}[theorem]{Remark}

\numberwithin{equation}{section}
\numberwithin{figure}{section}

\newcommand\Bcal{\mathcal{B}}

\newcommand\Pcal{\mathcal{P}}

\newcommand\Fcal{\mathcal{F}}

\newcommand\Dscr{\mathscr{D}}

\newcommand\Oscr{\mathscr{O}}

\newcommand\Gscr{\mathscr{G}}

\newcommand\C{\mathbb{C}}
\newcommand\D{\overline{\mathbb D}}
\newcommand\CP{\mathbb{CP}}
\renewcommand\D{\mathbb D}

\newcommand\N{\mathbb{N}}
\renewcommand\P{\mathbb{P}}

\newcommand\R{\mathbb{R}}

\newcommand\Z{\mathbb{Z}}

\renewcommand\b{\mathbb{B}}
\renewcommand\c{\mathbb{C}}

\newcommand\cp{\mathbb{CP}}
\renewcommand\d{\mathbb D}

\newcommand\n{\mathbb{N}}
\renewcommand\r{\mathbb{R}}
\newcommand\s{\mathbb{S}}

\newcommand\z{\mathbb{Z}}

\newcommand\igot{\mathfrak{i}}

\renewcommand\igot{\mathfrak{i}}

\newcommand\pgot{\mathfrak{p}}

\renewcommand\imath{\igot}

\newcommand\wt{\widetilde}

\newcommand\dist{\mathrm{dist}}

\newcommand\length{\mathrm{length}}

\newcommand\Flux{\mathrm{Flux}}

\newcommand\CMI{\mathrm{CMI}}

\def\dist{\mathrm{dist}}

\def\length{\mathrm{length}}
\def\Flux{\mathrm{Flux}}

\usepackage{color}

\begin{document}


\fancyhead[LO]{The Gauss map assignment for minimal surfaces}
\fancyhead[RE]{A.\ Alarc\'on and F.J.\ L\'opez}
\fancyhead[RO,LE]{\thepage}

\thispagestyle{empty}


\begin{center}
{\bf\Large On the Gauss map assignment for minimal surfaces
\\\smallskip  
and the Osserman 
curvature estimate
}

\medskip

%
%
{\bf Antonio Alarc\'on  and  Francisco J.\ L\'opez}
\end{center}

%
%
\medskip

\begin{quoting}[leftmargin={7mm}]
{\small
\noindent {\bf Abstract}\hspace*{0.1cm}
The Gauss map of a conformal minimal immersion of an open Riemann surface $M$ into $\R^n$, $n\ge 3$, is a holomorphic map $M\to{\bf Q}^{n-2}\subset \C\P^{n-1}$. Denote by ${\rm CMI}_{\rm full}(M,\R^n)$ and $\Oscr_{\rm full}(M,{\bf Q}^{n-2})$ the spaces of full conformal minimal immersions $M\to\R^n$ and full holomorphic maps $M\to{\bf Q}^{n-2}$, respectively, endowed with the compact-open topology. In this paper we show that the Gauss map assignment $\Gscr:{\rm CMI}_{\rm full}(M,\R^n)\to \Oscr_{\rm full}(M,{\bf Q}^{n-2})$,  taking a full conformal minimal immersion to its Gauss map, is an open map. This implies, in view of a result of Forstneri\v c and the authors, that $\Gscr$ is a quotient map. The same results hold for the map  $(\Gscr,\Flux):{\rm CMI}_{\rm full}(M,\R^n)\to \Oscr_{\rm full}(M,{\bf Q}^{n-2})\times H^1(M,\r^n)$, where $\Flux:{\rm CMI}_{\rm full}(M,\R^n)\to H^1(M,\r^n)$ is the flux assignment. 
 As application, we establish that the set of maps $G\in \Oscr_{\rm full}(M,{\bf Q}^{n-2})$ such that the family $\Gscr^{-1}(G)$ of all minimal surfaces in $\R^n$ with the Gauss map $G$ satisfies the classical Osserman curvature estimate, is meagre in the space of holomorphic maps $M\to {\bf Q}^{n-2}$.

\noindent{\bf Keywords}\hspace*{0.1cm} 
Minimal surface, Riemann surface, Gauss map, curvature estimate, Baire category theorem.


\noindent{\bf Mathematics Subject Classification (2020)}\hspace*{0.1cm} 
53C42, 
53A05, 
54E52. 
}
\end{quoting}


\section{Introduction and main results}\label{sec:intro}

\noindent
Throughout the paper, $n\ge 3$ is an integer, $M$ is an open Riemann surface, and ${\rm CMI}(M,\R^n)$ denotes the space of all conformal minimal immersions $M\to\R^n$ endowed with the compact-open topology. 
If $u=(u_1,\ldots,u_n)\in\CMI(M,\R^n)$, then the $(1,0)$-differential $\partial u=(\partial u_1,\ldots,\partial u_n)$ of $u$ is holomorphic and satisfies
\begin{equation}\label{eq:conformal}
	\partial u\neq 0 \;\text{ and }\; \sum_{j=1}^n (\partial u_j)^2=0 \;\text{ everywhere on $M$};
\end{equation} 
see e.g.\ \cite[Theorem 2.3.4]{AlarconForstnericLopez2021Book}.
It therefore determines the Kodaira-type holomorphic map $\Gscr(u)$ from $M$ to the hyperquadric
\[
	{\bf Q}^{n-2}=\big\{[z_1:\cdots:z_n]\in \C\P^{n-1}\colon z_1^2+\cdots+z_n^2=0\big\}\subset \C\P^{n-1}
\]
given by
\begin{equation}\label{eq:G(u)}
	\Gscr(u)(p)=[\partial u_1(p):\cdots:\partial u_n(p)],\quad p\in M.
\end{equation}
The map $\Gscr(u):M\to{\bf Q}^{n-2}$ is called the {\em generalized Gauss map}, or simply the {\em Gauss map}, of the conformal minimal immersion $u:M\to\R^n$. The Gauss map is of central interest in the classical theory of minimal surfaces. In particular, the fact that a conformal immersion $M\to\R^n$ is minimal if and only if its Gauss map is holomorphic (see \cite{Bonnet1860,Christoffel1867} and \cite[Theorem 1.1]{HoffmanOsserman1980}) has enabled to exploit complex-analytic methods to study minimal surfaces in Euclidean space, leading to many fundamental developments in the theory.  We refer to \cite[Chapter 12]{Osserman1986} and \cite[Chapter 5]{AlarconForstnericLopez2021Book} for historical background and further references. 
Conversely, if we are given a  $\c^n$-valued holomorphic $1$-form $\Phi=(\phi_1,\ldots,\phi_n)$ on $M$ such that $\Phi\neq 0$ and  $\sum_{j=1}^n \phi_j^2=0$  everywhere on $M$, and the real part $\Re \Phi$ is exact, then for any base point $p_0\in M$ and initial condition $x_0\in \r^n$ the map
\begin{equation}\label{eq:conversely}
u\colon M\to \r^n,\quad u(p)=x_0+\int_{p_0}^p \Re \Phi,
\end{equation}
is a conformal minimal immersion with $2\partial u=\Phi$ and   $\Gscr(u)=[\phi_1:\cdots:\phi_n]$. 

The {\em flux map} $\Flux(u)\in H^1(M,\r^n)\equiv{\rm Hom}(H_1(M,\Z),\R^n)$ of $u\in\CMI(M,\R^n)$ is the cohomology class given by 
\begin{equation}\label{eq:fluxxxx}
\Flux(u): H_1(M,\z)\to \r^n,\quad  \Flux(u)([C])=-2\imath \int_C\partial u;
\end{equation}
see e.g. \cite[Def.\ 2.3.2]{AlarconForstnericLopez2021Book}.
The flux is  the measure of obstruction to exactness of $\partial u$. From the physical point of view, its values along closed curves  are related to the surface tension or pressure forces of soap films.
We  endow $H_1(M,\Z)\cong \prod_{\gamma\in \Bcal} \z[\gamma]$, where $\Bcal$ is a  homology basis of $M$, with the product topology, and $H^1(M,\r^n)$ with the compact-open topology.
\subsection{The Gauss map assignment}\label{ss:TGMA}

Denote by $\Oscr(M,{\bf Q}^{n-2})$ the space of all holomorphic maps $M\to{\bf Q}^{n-2}$ endowed with the compact-open topology. 
By Cauchy estimates, it is easily seen that the {\em Gauss map assignment} 
\[
	\Gscr:\CMI(M,\R^n)\to \Oscr(M,{\bf Q}^{n-2}),
\]
sending an immersion $u\in \CMI(M,\R^n)$ to its Gauss map $\Gscr(u)\in \Oscr(M,{\bf Q}^{n-2})$ given in \eqref{eq:G(u)},
is a continuous map.
And it has been proved only recently that $\Gscr$ is surjective; that is, every holomorphic map $M\to {\bf Q}^{n-2}$ is the Gauss map of a conformal minimal immersion $M\to\R^n$ \cite[Theorem 1.1]{AlarconForstnericLopez2019JGA} (see also \cite[Theorem 5.4.1]{AlarconForstnericLopez2021Book}). It is thus reasonable to refer to the elements of $\Oscr(M,{\bf Q}^{n-2})$ as {\em Gauss maps} (of conformal minimal immersions $M\to\r^n$), and we will do so in this paper. 
The {\em flux assigment}
\[
 \Flux:\CMI(M,\R^n)\to H^1(M,\r^n),
\]
sending $u\in \CMI(M,\R^n)$ to its flux map $\Flux(u)\in H^1(M,\r^n)$ given in \eqref{eq:fluxxxx},
is easily seen to be continuous as well by the same arguments.

A holomorphic map $M\to\C\P^{n-1}$ is {\em full} if its image is not contained in any hyperplane. We denote by $\Oscr_{\rm full}(M,{\bf Q}^{n-2})\subset\Oscr(M,{\bf Q}^{n-2})$ the subspace of full maps. It is clear that $\Oscr_{\rm full}(M,{\bf Q}^{n-2})$ is open in $\Oscr(M,{\bf Q}^{n-2})$; it is also a dense subset since ${\bf Q}^{n-2}$ is an Oka manifold (see \cite[Example 4.4]{AlarconForstneric2014IM} or \cite[Example 5.6.2]{Forstneric2017E}). Likewise, an immersion $u\in\CMI(M,\R^n)$ is said to be {\em full} if its Gauss map $\Gscr(u)$ is full; 
and we denote by $\CMI_{\rm full}(M,\R^n)\subset \CMI(M,\R^n)$ the open subspace of full immersions (see \cite[Def.\ 2.5.2]{AlarconForstnericLopez2021Book}). It turns out that $\CMI_{\rm full}(M,\R^n)$ is dense in $\CMI(M,\R^n)$ by \cite[Theorem 5.3]{AlarconForstnericLopez2016MZ} (note that in \cite{AlarconForstnericLopez2016MZ} full conformal minimal immersions are called nondegenerate); see also \cite[Proposition 3.3.2 and Theorem 3.6.1]{AlarconForstnericLopez2021Book}.
The restriction
\[
	(\Gscr,\Flux):\CMI_{\rm full}(M,\R^n)\to \Oscr_{\rm full}(M,{\bf Q}^{n-2})\times H^1(M,\r^n)
\]
is continuous and surjective as well  (see \cite[Theorem 1.1]{AlarconForstnericLopez2019JGA} or \cite[Theorem 5.4.1]{AlarconForstnericLopez2021Book}). Furthermore, it has been recently proved that this map is in fact a Serre fibration \cite[Theorem 1.1]{AlarconLarusson2024Pisa}; that is, it satisfies the homotopy lifting property with respect to all CW-complexes. 

In this paper we delve into the study of the topological properties of the map $(\Gscr,\Flux)$ above by establishing its openness. In particular, we show that the Gauss map assignment for full immersions $\Gscr:\CMI_{\rm full}(M,\R^n)\to \Oscr_{\rm full}(M,{\bf Q}^{n-2})$
is an open map as well. Here is our main result.
%
%
\begin{theorem}\label{th:Gauss-F}
Let $M$ be an open Riemann surface and $n\ge 3$ an integer. Then the   map   
\begin{equation}\label{eq:G-F}
(\Gscr,\Flux):\CMI_{\rm full}(M,\R^n)\to \Oscr_{\rm full}(M,{\bf Q}^{n-2})\times H^1(M,\r^n)
\end{equation}
is an open  quotient map. In particular, the same result  holds true for the maps $\Gscr:\CMI_{\rm full}(M,\R^n)\to \Oscr_{\rm full}(M,{\bf Q}^{n-2})$ and $\Flux:\CMI_{\rm full}(M,\R^n)\to  H^1(M,\r^n)$.
\end{theorem}
In Section \ref{sec:intro-2} we explain a rather direct application of Theorem \ref{th:Gauss-F} concerning the classical curvature estimate for minimal surfaces introduced by Osserman in \cite{Osserman1960,Osserman1964}. We expect that this so basic topological property of the Gauss map assignment will lead to further applications other than those in this paper.

Let us say a word about the proof of Theorem \ref{th:Gauss-F}. Denote by
\begin{equation}\label{eq:null}
	{\bf A}_*^{n-1}=\big\{ z\in\c^n_*=\c^n\setminus\{0\}\colon z_1^2+\cdots+z_n^2=0\big\}
\end{equation}
the punctured null quadric in $\C^n$, by $\Oscr(M,{\bf A}_*^{n-1})$ the space of holomorphic maps $M\to {\bf A}_*^{n-1}$ endowed with the compact-open topology, and by $\Oscr_{\rm full}(M,{\bf A}_*^{n-1})$ the subspace of full maps, that is, with the image not contained in any hyperplane of $\C^n$.
Also let $\pi:{\bf A}_*^{n-1}\to{\bf Q}^{n-2}$ be the restriction of the canonical projection $\pi:\C^n_*\to\C\P^{n-1}$, and define 
$\pi_*\colon  \Oscr(M,{\bf A}_*^{n-1})\to  \Oscr(M,{\bf Q}^{n-2})$ by
\begin{equation}\label{eq:pi*}
	\quad \pi_*(f)=\pi \circ f=[f_1:\cdots:f_n],\quad f=(f_1,\ldots,f_n)\in \Oscr(M,{\bf A}_*^{n-1}).
\end{equation}
Fix a holomorphic $1$-form $\theta$ on $M$ vanishing nowhere (see \cite{GunningNarasimhan1967MA}), and recall that $\partial u/\theta\in \Oscr(M,{\bf A}_*^{n-1})$ for all $u\in \CMI(M,\R^n)$; see \eqref{eq:conformal} and \eqref{eq:null}. We can thus consider the map $\Dscr: \CMI(M,\R^n)\to \Oscr(M,{\bf A}_*^{n-1})$ given by
\begin{equation}\label{eq:Psi(u)}
	\Dscr(u)=\partial u/\theta,\quad u\in \CMI(M,\R^n).
\end{equation}
Then the Gauss map assignment for full immersions  factorizes in a natural way as
\[
	\Gscr=\pi_*\circ\Dscr:\CMI_{\rm full}(M,\R^n)\to \Oscr_{\rm full}(M,{\bf Q}^{n-2}),
\]
where the domain and range of $\Dscr$ and $\pi_*$ are restricted to the corresponding subspaces of full maps. Therefore, we have the following diagram.
\begin{equation}\label{eq:diagram*}
\xymatrix{
	\CMI_{\rm full}(M, \r^n)  \ar@{->}[rrr]^{ (\Dscr\,,\, \Flux)\;\;\;\;\quad}  \ar@{->}[ddrrr]_{ (\Gscr\,,\, \Flux)} 
& &	&   \Oscr_{\rm full}(M,{\bf A}_*^{n-1})\times  H^1(M,\r^n)   \ar@{->}[dd]^{ \pi_*\times {\rm Id}} \\
& & & \\
&	&   &   \Oscr_{\rm full}(M,{\bf Q}^{n-2})\times H^1(M,\r^n).
}
\end{equation}
The proof of Theorem \ref{th:Gauss-F} then consists of two steps. 
In the first one we prove that the map $\pi_*$ is open; see Corollary \ref{co:Gauss} and the more precise result in Proposition \ref{pr:CCP}. However, the map $\Dscr$ is clearly not, except in the case when $M$ is simply connected, due to the period problem; hence the same holds for $(\Dscr, \Flux)$. In the second step of the proof we overcome this difficulty by showing that for any open set $U\subset  \CMI_{\rm full}(M,\R^n)$ and any immersion $u\in U$ there exist neighborhoods $V$ of $\Dscr(u)$ in $\Oscr_{\rm full}(M,{\bf A}_*^{n-1})$ and $W$ of $\Flux(u)$ in $H^1(M,\r^n)$ such that $\pi_*(V)\times W\subset (\Gscr,\Flux)(U)$; see Proposition \ref{pr:Psi}. These two facts together trivially lead to the openness of $(\Gscr,\Flux)$.

Since $\Oscr_{\rm full}(M,{\bf Q}^{n-2})$ is an open subset of $\Oscr(M,{\bf Q}^{n-2})$, Theorem \ref{th:Gauss-F} ensures that the map $(\Gscr,\Flux):\CMI_{\rm full}(M,\R^n)\to \Oscr(M,{\bf Q}^{n-2})\times H^1(M,\r^n)$ is open as well. 
However, Theorem \ref{th:Gauss-F} does not seem to imply that the unrestricted map $(\Gscr,\Flux):\CMI(M,\R^n)\to \Oscr(M,{\bf Q}^{n-2})\times H^1(M,\r^n)$ be an open map, which remains an open question. The same holds for the component  maps $\Gscr:\CMI(M,\R^n)\to \Oscr(M,{\bf Q}^{n-2})$ and $\Flux:\CMI(M,\R^n)\to  H^1(M,\r^n)$. Indeed, due to technical difficulties, our approach does not adapt to the situation when the given conformal minimal immersion $u:M\to\R^n$ in the second step of the proof is not full. 

\subsection{The Osserman curvature estimate}\label{sec:intro-2}
For an immersed (connected) surface $\varphi:\Sigma\to\R^n$ and a point $p\in\Sigma$, we denote by  $K_\varphi$ the Gauss curvature function of $\Sigma$ and by  $d_\varphi(p)$ the geodesic distance from $p$ to the ideal boundary of $\Sigma$, that is, the infimum of the intrinsic lengths of divergent paths in $\Sigma$ leaving from $p$. Note that $\varphi$ is complete if and only if $d_\varphi(p)=+\infty$ for some (hence for all) $p\in\Sigma$.

An immersed minimal surface $\varphi: \Sigma \to \r^n$ is {\em nonflat} if it is not contained in any affine plane, or equivalently, if  $K_\varphi\not\equiv 0$. A family $\Fcal$ of nonflat immersed minimal surfaces in $\R^n$ satisfies the {\em Osserman curvature estimate} (see \cite[Theorem 1]{Osserman1960} or \cite[Theorem 1.2]{Osserman1964})
if there is a constant $C=C(\Fcal)>0$ such that
\begin{equation}\label{eq:Ke}
	|K_\varphi(p)|\,d_\varphi(p)^2\le C
\end{equation}
 for every  $\varphi:\Sigma\to\R^n$ in $\Fcal$ and every point $p\in \Sigma$.
Note that if $\Fcal$ satisfies the Osserman curvature estimate then it contains no complete surface. Since the expression in the left hand side of \eqref{eq:Ke} is invariant under composition of  $\varphi$ with homotheties and translations, it is natural to assume that the family $\Fcal$ is closed under these transformations. If this is the case, then the family $\Fcal$ satisfies the Osserman curvature estimate if and only if \eqref{eq:Ke} holds for every  $\varphi:\Sigma\to\R^n$  in  $\Fcal$ and every point $p\in\Sigma$ with $K_\varphi(p)=-1$. This curvature estimate plays a crucial role in some fundamental results in the classical theory of minimal surfaces which depend only on the Gauss map. For instance, Fujimoto 
\cite{Fujimoto1988JMSJ,Fujimoto1992} and Osserman and Ru  
 \cite{OssermanRu1997} showed that the family of nonflat minimal surfaces in $\R^n$ whose Gauss maps omit a given set of $k>n(n+1)/2$ hyperplanes  in $\cp^{n-1}$ located in general position satisfies the Osserman curvature estimate. This shows that the only complete minimal surfaces in $\R^n$ with the Gauss map omitting $k$ such hyperplanes are the planes (see \cite{Osserman1960,Osserman1964} again for previous partial results). Likewise, Schoen 
\cite{Schoen1983}  established that the family of nonflat stable minimal surfaces in $\R^3$ satisfies the Osserman curvature estimate, which implies that the planes are the only complete stable minimal surfaces in $\R^3$ (see \cite{Fischer-ColbrieSchoen1980,Pogorelov1981,doCarmoPeng1997} for alternative proofs of this and recall that stability depends only on the Gauss map of the minimal surface). We refer to Ros \cite{Ros2022} for further discussion on the Osserman curvature estimate and its relation with the Gauss map. 

It is therefore natural to wonder when a set of Gauss maps $\varnothing\neq Y\subset \Oscr(M,{\bf Q}^{n-2})$ has the property that the family $\Gscr^{-1}(Y)\subset \CMI(M,\R^n)$ of all minimal surfaces in $\R^n$ with the Gauss map in $Y$ satisfies the Osserman curvature estimate. 
Recall that $\Gscr^{-1}(Y)$ is always nonempty \cite[Theorem 1.1]{AlarconForstnericLopez2019JGA} (see also \cite[Theorem 5.4.1]{AlarconForstnericLopez2021Book}), and it is clearly closed under scalings and translations. In this paper we focus in the most relevant case when the set $Y$ consists of a single map. 

%
%
\begin{definition}\label{def:pseudo}
Let $M$ be an open Riemann surface, $n\ge 3$ an integer, and $G\in  \Oscr_{\rm full}(M,{\bf Q}^{n-2})$ a full Gauss map.

(a) We say that $G$ satisfies the {\em Osserman curvature estimate at $p\in M$} if
\begin{itemize} 
\item $p$ is a critical point of $G$ and $\Gscr^{-1}(G)$ contains no complete immersion, or
\smallskip
\item  $p$ is a noncritical point of $G$ and there is a constant $C=C(G,p)>0$ with
\begin{equation}\label{eq:Kudu}
	|K_u(p)|\,d_u(p)^2\le C\quad \text{for all $u\in\Gscr^{-1}(G)$.}
\end{equation}
  \end{itemize}
%
We denote by $\Oscr^{K,p}_{\rm full}(M,{\bf Q}^{n-2})$ the set of Gauss maps in $\Oscr_{\rm full}(M,{\bf Q}^{n-2})$ satisfying the Osserman curvature estimate at the point $p$, and $\Oscr^{^\neg K,p}_{\rm full}(M,{\bf Q}^{n-2})=\Oscr_{\rm full}(M,{\bf Q}^{n-2})\setminus \Oscr^{K,p}_{\rm full}(M,{\bf Q}^{n-2})$.

(b) We say that $G$ satisfies the {\em Osserman curvature estimate} if the family of full immersed minimal surfaces $\Gscr^{-1}(G)$ does, that is, there is a constant $C=C(G)>0$ such that
\eqref{eq:Kudu} holds
for every conformal minimal immersion $u\in \Gscr^{-1}(G)$ and every point $p\in M$. We denote by $\Oscr^K_{\rm full}(M,{\bf Q}^{n-2})$ the set of Gauss maps in $\Oscr_{\rm full}(M,{\bf Q}^{n-2})$ satisfying the Osserman curvature estimate, and $\Oscr^{^\neg K}_{\rm full}(M,{\bf Q}^{n-2})=\Oscr_{\rm full}(M,{\bf Q}^{n-2})\setminus \Oscr^K_{\rm full}(M,{\bf Q}^{n-2})$.
\end{definition}
It is clear that 
\begin{equation}\label{eq:inclu-example}
	\Oscr^K_{\rm full}(M,{\bf Q}^{n-2})\subset \bigcap_{p\in M}\Oscr^{K,p}_{\rm full}(M,{\bf Q}^{n-2}),
\end{equation}
and so, passing to complements,
\begin{equation}\label{eq:complements}
	\bigcup_{p\in M}\Oscr^{^\neg K,p}_{\rm full}(M,{\bf Q}^{n-2})\subset \Oscr^{^\neg K}_{\rm full}(M,{\bf Q}^{n-2}).
\end{equation}
In Section \ref{sec:Example}, we give an example (with $M=\D$ and $n=3$) which shows that the inclusion \eqref{eq:inclu-example} is not an equality in general. 
Recall that an immersion $u\in \CMI(M,\R^n)$ satisfies $K_u(p)=0$ if and only if $p$ is a critical point of $\Gscr(u)$ (see \cite[Eq.\ (3.8)]{HoffmanOsserman1980}).   Condition \eqref{eq:Kudu} in  Definition \ref{def:pseudo} (a) ensures  that if $G\in\Oscr^{K,p}_{\rm full}(M,{\bf Q}^{n-2})$ and  $p$ is a noncritical point of $G$ then  $\Gscr^{-1}(G)$ contains no complete immersion.
An interesting question is to understand how $\Oscr^{K,p}_{\rm full}(M,{\bf Q}^{n-2})$ depends on the point $p\in M$, even locally. Roughly speaking, a map $G\in  \Oscr_{\rm full}(M,{\bf Q}^{n-2})$  satisfies the Osserman curvature estimate if it is far from being the Gauss map of a complete conformal minimal surface $M\to \r^n$.

Note that $\Oscr(M,{\bf Q}^{n-2})$ is completely metrizable 
 and so a Baire space by the Baire Category Theorem. 
 Our second main result
 shows that the set of Gauss maps in $\Oscr(M,{\bf Q}^{n-2})$ which satisfy the Osserman curvature estimate (Definition \ref{def:pseudo}) is meagre, that is, {\em thin} or {\em negligible} in Baire category sense.

\begin{theorem}\label{th:Gmeagre}
Let $M$ be an open Riemann surface and $n\ge 3$. Then the set $\Oscr^{^\neg K,p}_{\rm full}(M,{\bf Q}^{n-2})$ is residual in $\Oscr(M,{\bf Q}^{n-2})$ for every $p\in M$. In particular, a generic Gauss map in $\Oscr(M,{\bf Q}^{n-2})$ is full and does not satisfy the Osserman curvature estimate.
\end{theorem}
Recall that a set in a Baire space is
%
%
{\em residual} 
if it contains a dense set that can expressed as a countable intersection of open sets, and is {\em meagre} (or {\em of first category}) 
if its complement is residual.
It is customary to say that a property of elements in a Baire space is {\em generic}, or that a {\em generic} element satisfies the property, if the property holds on a residual set.
In other words, generic properties are those enjoyed by {\em almost all} elements in Baire category sense; see, e.g., \cite[Sec.\ 25]{Willard1970}. Using this terminology, Theorem \ref{th:Gmeagre} (see Theorem \ref{th:residual} for a more precise statement) ensures that if we are given a point $p\in M$ then a generic Gauss map  $G\in  \Oscr(M,{\bf Q}^{n-2})$ satisfies that $dG_p\neq 0$ and 
\[
\sup\{d_u(p)\colon u\in \Gscr^{-1}(G),\, K_u(p)=-1\}=+\infty.
\]
The final assertion in the theorem is then guaranteed by \eqref{eq:complements}.
So, roughly speaking, almost all maps in $\Oscr(M,{\bf Q}^{n-2})$ are the Gauss map of a full conformal minimal surface $M\to\R^n$ which is, under curvature normalization at a given point, as close as desired of being complete. 
It is a long-standing open question to recognize which holomorphic maps are the Gauss map of a complete minimal surface; we refer to \cite[Chapter 5]{AlarconForstnericLopez2021Book} for background and references and to \cite{AlarconLarusson2024Pisa} for recent progresses from a different viewpoint. Denote by $\Oscr^c_{\rm full}(M,{\bf Q}^{n-2})$ the set of maps in $\Oscr_{\rm full}(M,{\bf Q}^{n-2})$ which are the Gauss map of a complete full conformal minimal immersion $M\to\r^n$. By Definition \ref{def:pseudo}, we have
\begin{equation}\label{eq:complete}
	\Oscr^{\rm c}_{\rm full}(M,{\bf Q}^{n-2})\subset \bigcap_{p\in M} \Oscr^{^\neg K,p}_{\rm full}(M,{\bf Q}^{n-2}).
\end{equation}
In Section \ref{sec:example-2} we give an example (also with $n=3$) that shows that the inclusion \eqref{eq:complete} is not an equality in general; note that every such has no critical points. The set $\Oscr^{\rm c}_{\rm full}(M,{\bf Q}^{n-2})$ is dense in $\Oscr(M,{\bf Q}^{n-2})$ by \cite[Theorem 5.6]{AlarconForstnericLopez2016MZ} and \cite[Theorem 1.1]{AlarconForstnericLopez2019JGA} (see also \cite[Theorems 5.4.1 and 3.9.1]{AlarconForstnericLopez2021Book}). Furthermore, the inclusion $\Oscr^{\rm c}_{\rm full}(M,{\bf Q}^{n-2})\hookrightarrow\Oscr(M,{\bf Q}^{n-2})$ is a weak homotopy equivalence, and a homotopy equivalence if $M$ is of finite topology, by  \cite[Theorem 1.2]{AlarconLarusson2024Pisa}. Moreover, $\Oscr^{^\neg K,p}_{\rm full}(M,{\bf Q}^{n-2})$ is a residual set in $\Oscr(M,{\bf Q}^{n-2})$ for every $p\in M$ by Theorem \ref{th:Gmeagre}, while a generic minimal surface in the completely metrizable space $\CMI(M,\R^n)$ is full and complete by \cite[Theorem 1.2]{AlarconLopez2024-2}. However, the question whether a generic map in $\Oscr(M,{\bf Q}^{n-2})$ is the Gauss map of a complete full conformal minimal immersion in $\CMI(M,\R^n)$ remains open.
%
%
%

\subsection*{Organization of the paper} We explain the first step of the proof of Theorem \ref{th:Gauss-F} in Section \ref{sec:pi*}. We then continue with the second step and complete the proof of the theorem in Section \ref{sec:Gauss}. Section \ref{sec:Gmeagre} is dedicated to the results concerning the Osserman curvature estimate. We prove Theorem \ref{th:Gmeagre} in Section \ref{sec:proofT-2}, and show by examples that the inclusions \eqref{eq:inclu-example} and \eqref{eq:complete} are in general proper in Sections \ref{sec:Example} and \ref{sec:example-2}. Finally, we point out in Section \ref{sec:homotopy} that the inclusions 
\[
\Oscr^{^\neg K,p}_{\rm full}(M,{\bf Q}^{n-2}) \hookrightarrow
	\Oscr^{^\neg K}_{\rm full}(M,{\bf Q}^{n-2}) \hookrightarrow
	 \Oscr(M,{\bf Q}^{n-2}),\quad p\in M,
\] 
are weak homotopy equivalences, thereby determining the rough shape of the first two spaces.


\section{Openness of $\pi_*\colon  \Oscr(M,{\bf A}_*^{n-1})\to  \Oscr(M,{\bf Q}^{n-2})$}\label{sec:pi*}

\noindent
In this section we explain the first step in the proof of Theorem \ref{th:Gauss-F}. It is provided by Corollary \ref{co:Gauss}. We begin with the following result of independent interest. Recall that $\pi:\C^n_*\to\C\P^{n-1}$ denotes the canonical projection given by $\pi(z_1,\ldots,z_n)=[z_1:\cdots:z_n]$.

%
%
\begin{proposition}\label{pr:CCP}
Let $M$ be an open Riemann surface and $n\ge 3$ an integer. Then the map  
\[
	\pi_*:\Oscr(M,\C^n_*)\to  \Oscr(M,\CP^{n-1}), \quad
	\pi_*(f)=\pi\circ f,
\]
is continuous, open, and surjective, hence a quotient map.
\end{proposition}

We start with some preparations. Throughout the paper, we shall denote by $|\cdot|$, $\dist(\cdot,\cdot)$, and $\length(\cdot)$ the standard Euclidean norm, distance, and length in $\R^d$, $d\in\N$. 
If $X$ is a topological space, then for a continuous map $f:X\to\R^d$ and a compact set $L\subset X$, we denote by $\|f\|_L=\sup_{x\in L}|f(x)|$ the maximum norm of $f$ on $L$. 
 We shall use the standard multiplicative notation for divisors; in particular, $(f)$ denotes the divisor of a nonzero meromorphic function $f$ on a subset of an open Riemann surface.
Recall that for a topological space $X$,  the space ${\rm Div}_m(X)$ of integral divisors of order $m\ge 0$ with support in $X$ is the quotient of the  $m$-fold Cartesian product $X^m=X\times \stackrel{m}{\ldots} \times X$ under the natural action of the group of permutations of $\{1,\ldots,m\}$; here $X^0=\{1\}$. The set ${\rm Div}(X)=\bigcup_{m\geq 0} {\rm Div}_m(X)$ is an abelian multiplicative group  in a natural way.

We expect  the following to have been observed before, but we do not  know a reference for it.
%
%
\begin{claim}\label{cl:mero}
Let $M$ be an open Riemann surface,  $L\subset M$ a smoothly bounded compact domain, $m\geq 0$, and $E_0\in {\rm Div}_m(\mathring L)$ a divisor. Then, for any $\epsilon>0$ there is a neighborhood $W$ of $E_0$ in ${\rm Div}_m(\mathring L)$ satisfying  that  for any $E\in W$ there is a meromorphic function $\Psi_E$ on $L$ with divisor  $(\Psi_E)=E E_0^{-1}$ such that
 $\|\Psi_E-1\|_{b L}<\epsilon$.
\end{claim}
Since $\C\setminus\{0\}$ is an Oka manifold (see \cite{Forstneric2017E} for a monograph in Oka theory), if $L$ is Runge\footnote{A compact set in an open Riemann surface is {\em Runge} if its complement has no relatively compact connected components.} in $M$ then the Runge approximation theorem with jet interpolation (see e.g. \cite[Theorem 1.13.3]{AlarconForstnericLopez2021Book}) enables us to assume that $\Psi_E$ is holomorphic and nowhere vanishing on $M\setminus{\rm supp}(EE_0^{-1})$, hence its divisor on $M$ still equals $EE_0^{-1}$.
%
%
\begin{proof} If $m=0$ then ${\rm Div}_m(\mathring L)=\{1\}$ and the conclusion of the claim is trivial.
Assume that $m>0$ and write $E_0=\prod_{j=1}^k p_j^{\nu_j}$, where $p_i\neq p_j$  if $i\neq j$, and $\nu_j>0$ for all $j$, and $\sum_{j=1}^k\nu_j=m$. For each $j\in \{1,\ldots,k\}$
we choose a holomorphic function $h_j$ on $L$ having a zero of order $1$ at $p_j$ and vanishing nowhere on $L\setminus\{p_j\}$.
 We can then take a smoothly bounded closed disc $D_j\subset \mathring L$ such that $p_j\in \mathring D_j$, 
 $h_j\colon D_j\to \c$ is injective, and   $h_j^{-1}(h_j(D_j))=D_j$.  Make sure  that $D_1,\ldots,D_k$ are pairwise disjoint.
For each $q\in \mathring D_j$ let $f_j[q]$ denote the meromorphic function on $L$ defined by
\[
f_j[q]=\frac{h_j- h_j(q)}{h_j},
\]
and observe that $(f_j[q])=q\, p_j^{-1}$.
Let  $W$ be the natural neighborhood of $E_0$ in ${\rm Div}_m(\mathring L)$ given by the projection of  $\mathring D_1^{\nu_1}\times\cdots\times \mathring D_k^{\nu_k}$, where  $D_j^{\nu_j}=D_j\times \stackrel{\nu_j}{\ldots} \times D_j$.  For each $E\in W$ write $E=E_1\cdots E_k$, where $E_j=\prod_{i=1}^{\nu_j}q_{i,j}\in {\rm Div}_{\nu_j}(\mathring D_j)$, $j=1,\ldots,k$, and denote by $\Psi_E$ the meromorphic function on $L$ given by
\[
	\Psi_E=\prod_{j=1}^k \prod_{i=1}^{\nu_j} f_j[q_{i,j}].
\]
 It is clear that $(\Psi_E)=EE_0^{-1}$. Finally, if  the discs $D_j$, $j=1,\ldots,k$, are chosen sufficiently small then  $\|\Psi_E-1\|_{b L}<\epsilon$ for all $E\in W$.
 \end{proof}
 %
 %

%
%
\begin{proof}[Proof of Proposition \ref{pr:CCP}]  
It is clear that $\pi_*$ is continuous, and it is known to be surjective; see, e.g., \cite[Proposition 5.4.4]{AlarconForstnericLopez2021Book}. Let us see that $\pi_*$ is open.
Fix  the Fubini-Study metric on $\c\P^{n-1}$ and recall that its induced distance function is given by the expression
\begin{equation}\label{eq:fubini}
\dist(z,w)=\arccos \Big( \frac{|z\cdot \overline w|}{|z| |w|}\Big) \quad \text{for all $z,w\in \c\P^{n-1}$,}
\end{equation}
where $z\cdot  w=\sum_{j=1}^n z_j  w_j$ for all $z,w\in \c^n$; see \cite[p.\ 116, Eq.\ (4.45) and (4.46)]{BengtssonZyczkowski2006}. (To to the computation in the right hand side of \eqref{eq:fubini}, one can choose $z,w\in\C^n\setminus\{0\}$ to be any representatives of the points $z,w\in\C\P^{n-1}$ in the left hand side, which justifies the abuse of notation.)
Recall that this metric induces the standard topology in $\C\P^{n-1}$.
Let $f=(f_1,\ldots,f_n)\in  \Oscr(M,\C^n_*)$, let $L\subset M$ be a connected, smoothly bounded compact domain, and fix $\epsilon>0$. Set
\begin{equation}\label{eq:U-neigh}	
	U=\big\{h\in \Oscr(M,\C^n_*)\colon \|h-f\|_L<\epsilon\big\}.
\end{equation}
To complete the proof it suffices to check that the set $\pi_*(U)$ is a neighborhood of $\pi_*(f)$ in $\Oscr(M,\CP^{n-1})$, which is seen by finding a number $\delta>0$ such that the set
\begin{equation}\label{eq:Vdelta}
V_\delta=\{G\in \Oscr(M,\CP^{n-1})\colon \sup_L\dist(G,\pi_*(f))<\delta\}
\end{equation}
 is contained in $\pi_*(U)$; recall that $ \Oscr(M,\C^n_*)$ and $\Oscr(M,\CP^{n-1})$ are both endowed with the compact-open topology.   We may assume that $f_1\colon M\to \c$ is not identically zero (otherwise the role of $f_1$ in the proof would be played by a different component function of $f$),   $L$ is Runge in $M$, and $b L \cap f_1^{-1}(0)=\varnothing$ (otherwise we replace $L$ by a larger domain satisfying these properties; recall that $f_1$ is holomorphic).  Since $\pi_*$ is surjective and  $\pi_*(\varphi g)=\pi_*(g)$ holds for every holomorphic function $\varphi:M\to\C_*=\C\setminus\{0\}$ and every $g\in \Oscr(M,\C^n_*)$, it then suffices to prove the following.
%
%
\begin{claim}\label{cl:Gauss} 
There exists a number $\delta>0$ satisfying that for any $g\in  \pi_*^{-1}(V_\delta)
$ there is a nowhere vanishing holomorphic function $\varphi_g$ on $M$ such that $\varphi_g g\in U$.
\end{claim}

In order to prove the claim, we distinguish cases.

\noindent{\em Case 1.} Assume  that $f_1$ has some zeroes on $L$. Since $f_1$ vanishes nowhere on $bL$, we may write $(f_1|_{L})=\prod_{j=1}^k p_j^{\nu_j}\in {\rm Div}_m(\mathring L)$, where $p_i\neq p_j$ if $i\neq j$, $\nu_j>0$ for all $j$,  and $m=\sum_{j=1}^k \nu_j$. Thus, ${\rm supp}((f_1|_{L}))=\{p_1,\ldots,p_k\}\subset\mathring L$. Fix a number
\begin{equation}\label{eq:ep0}
	0<\epsilon_0<\min\big\{1\,,\,\frac{\epsilon}{3\|f\|_{bL}}\big\}.
\end{equation}
Claim \ref{cl:mero} furnishes us with a neighborhood $W$ of $(f_1|_{L})$ in ${\rm Div}_m(\mathring L)$  satisfying  that  for any divisor $E\in W$ there is a meromorphic function $\Psi_E$ on $L$  such that
 \begin{equation}\label{eq:claim}
 	(\Psi_E)=E(f_1|_{L})^{-1} \quad\text{and}\quad \|\Psi_E-1\|_{b L}<\epsilon_0.
\end{equation}
Let $D_j$ be a smoothly bounded  compact disc neighborhood of $p_j$ in $\mathring L$, $j=1,\ldots,k$. Choose the discs $D_1, \ldots,D_k$ to be  pairwise disjoint and so small that the projection of the set  $D_1^{\nu_1}\times \cdots \times D_k^{\nu_k}\subset (\mathring L)^m$ into 
 ${\rm Div}_m(\mathring L)$ is  contained in $W$. Call 
\begin{equation}\label{eq:D}
	D=\bigcup_{j=1}^k D_j \subset \mathring L.
\end{equation}
Fix a number $0<\delta<\pi/2$ to be specified later.  By \eqref{eq:fubini} and \eqref{eq:Vdelta},  we have that 
\[
 \frac{|g\cdot \overline f|}{|g| |f|}=\big\langle \frac{\overline g \cdot f}{|\overline g\cdot f|} \frac{g}{|g|},  \frac{f}{|f|}  \big\rangle >\cos \delta>0 \quad \text{everywhere on $L$ for all $g\in \pi_*^{-1}(V_\delta)$,}
\]
 where $\langle z,w\rangle=\Re (z\cdot \overline w)$ denotes the Euclidean scalar product in $\c^n\equiv \r^{2n}$; observe that $\frac{\overline g \cdot f}{|\overline g\cdot f|} \frac{g}{|g|}\cdot  \frac{\overline f}{|f|}$ is a positive real number.   Since $ \frac{\overline g \cdot f}{|\overline g\cdot f|} \frac{g}{|g|}$ and $\frac{f}{|f|}$ are unitary, the above inequality implies that
\[
\Big\| \frac{\overline g \cdot f}{|\overline g\cdot f|} \frac{g}{|g|} -\frac{f}{|f|}\Big\|_{L}<2 \sin(\delta/2),
\]
 and so 
\begin{equation}\label{eq:delta'}
\Big\| \frac{\overline g \cdot f}{|\overline g\cdot f|} \frac{|f|}{|g|}g -f\Big\|_{L}<2\|f\|_{L} \sin(\delta/2) \quad \text{for all $g\in \pi_*^{-1}(V_\delta)$.}
\end{equation} 
Since $f_1$ vanishes nowhere on the compact set $L\setminus \mathring D=L\setminus\bigcup_{j=1}^k\mathring D_j$ (see \eqref{eq:D}), equation  \eqref{eq:delta'} enables us to choose  $\delta>0$  so small  that 
 \begin{equation}\label{eq:g1D}
 \text{$g_1$ vanishes nowhere on  $L \setminus \mathring D$  for all  $g=(g_1,\ldots,g_n)\in \pi_*^{-1}(V_\delta)$.}
\end{equation} 
We can also assume in view of   \eqref{eq:delta'} and \eqref{eq:g1D}  that    $\delta>0$ is  so  small  that the planar curves  $\big(\frac{\overline g \cdot f}{|\overline g\cdot f|} \frac{|f|}{|g|}g_1\big) \circ \gamma$ and $f_1 \circ \gamma $ have the same  winding number about the origin for every component  $\gamma\in H_1(L \setminus \mathring D,\z)$ of $\partial D$ and    $g=(g_1,\ldots,g_n)\in \pi_*^{-1}(V_\delta)$. Since $\frac{\overline g \cdot f}{|\overline g\cdot f|} \frac{|f|}{|g|}$ does not vanish on $L$ and each component of $D$ is simply-connected, we have that  $\Big(\frac{\overline g \cdot f}{|\overline g\cdot f|} \frac{|f|}{|g|}g_1\Big)\circ \gamma $ and $g_1\circ \gamma $ have the same winding number about the origin as well.   
 This implies that  $g_1 $ and $f_1$ have the same number of zeroes, counted with multiplicity, in each component of $\mathring D$  for all  $g\in \pi_*^{-1}(V_\delta)$.  It follows that   $(g_1|_{L})\in W$, hence the meromorphic function $\phi_g=\Psi_{(g_1|_{L})}$ on $L$  satisfies 
 \begin{equation}\label{eq:phig}
(\phi_g)=(g_1|_{L})(f_1|_{L})^{-1}\quad\text{and}\quad  
 \|\phi_g-1\|_{b L}<\epsilon_0
\end{equation}
  for all  $g=(g_1,\ldots,g_n)\in \pi_*^{-1}(V_\delta)$; see \eqref{eq:ep0} and \eqref{eq:claim}. In particular, the function 
\begin{equation}\label{eq:varphig}
  	\varphi_g=  \phi_g \,\frac{f_1}{g_1}
\end{equation}
is holomorphic and vanishes nowhere on $L$. By \eqref{eq:delta'}   it is clear that
\[
\Big\| \frac{\overline g \cdot f}{|\overline g\cdot f|} \frac{|f|}{|g|}g_1 -f_1\Big\|_{L}<2\|f\|_{L} \sin(\delta/2) \quad \text{for all $g\in \pi_*^{-1}(V_\delta)$.}
\]
Taking into account that $f_1$ vanishes nowhere on the compact set $L\setminus \mathring D$, this inequality, \eqref{eq:delta'},  and  \eqref{eq:g1D} enable us to assume that $\delta>0$ is so small that
\[
 \Big\| \frac{g }{g_1}  - \frac{f }{f_1} \Big\|_{L \setminus \mathring D} <\frac{\epsilon}{3  \|f_1\|_{L \setminus \mathring D} } \quad \text{for all $g\in \pi_*^{-1}(V_\delta)$},
 \]
and hence
 \begin{equation}\label{eq:e2}
\Big\| \frac{f_1 }{g_1}g  -f\Big\|_{L \setminus \mathring D}\leq \|f_1\|_{L \setminus \mathring D} \Big\| \frac{g }{g_1}  - \frac{f }{f_1} \Big\|_{L \setminus \mathring D} <\frac{\epsilon}3 \quad \text{for all $g\in \pi_*^{-1}(V_\delta)$.}
 \end{equation}
It then follows from the maximum modulus principle, \eqref{eq:ep0}, \eqref{eq:phig}, \eqref{eq:varphig}, and \eqref{eq:e2} that
\begin{eqnarray*}
\left\| \varphi_g g  -f\right\|_{L } & \leq &  \left\| \varphi_g g  -f\right\|_{bL }
\\
& \le &  \|\phi_g-1\|_{bL }\Big( \Big\| \frac{f_1}{g_1}g-f \Big\|_{bL } + \|f\|_{bL } \Big) +  \Big\| \frac{f_1}{g_1}g-f \Big\|_{bL }
\\
& < & \epsilon_0\big(\frac{\epsilon}3+\|f\|_{bL }\big)+\frac{\epsilon}3 \; < \; \epsilon\quad
\text{for all $g\in \pi_*^{-1}(V_\delta)$.}
\end{eqnarray*}
Finally, since $\C_*$ is an Oka manifold and $L \subset M$ is a Runge compact subset, the Runge approximation theorem (see e.g. \cite[Theorem 1.13.3]{AlarconForstnericLopez2021Book}) enables us to assume that the function $\varphi_g$ is holomorphic and nowhere vanishing on $M$, and hence $\varphi_g g\in U$ for all $g\in \pi_*^{-1}(V_\delta)$; see \eqref{eq:U-neigh} and \eqref{eq:Vdelta}. This shows Claim \ref{cl:Gauss} in the case when $f_1$ has some zeroes in $L$.

\noindent{\em Case 2.} Assume that $f_1$ vanishes nowhere on $L$. A simplification of the arguments in Case 1 then provides a number $\delta>0$ such that $g_1$ vanishes nowhere on $L$ and $\|\frac{f_1}{g_1}g-f\|_{L}<\epsilon$ for all $g=(g_1,\ldots,g_n)\in \pi_*^{-1}(V_\delta)$; use \eqref{eq:delta'} and cf.\ \eqref{eq:e2}. By Runge's theorem, we may approximate $f_1/g_1$ uniformly on $L$ by a holomorphic function $\varphi_g:M\to\C_*$. If the approximation is close enough, then $\|\varphi_gg-f\|_{L}<\epsilon$ and hence $\varphi_gg\in U$ for all $g\in \pi_*^{-1}(V_\delta)$; see \eqref{eq:U-neigh} and \eqref{eq:Vdelta}. This completes the proof of Claim \ref{cl:Gauss}, thereby proving Proposition \ref{pr:CCP}.
\end{proof}

%
%
\begin{corollary}\label{co:Gauss}
Let $M$ be an open Riemann surface and $n\ge 3$ an integer. Then the map  
\[
	\pi_*:  \Oscr(M,{\bf A}_*^{n-1})\to  \Oscr(M,{\bf Q}^{n-2})
\]
given by \eqref{eq:pi*}
is continuous, open, and surjective, hence a quotient map.
\end{corollary}
\begin{proof}
We know that $\pi_*:\Oscr(M,\C^n_*)\to  \Oscr(M,\CP^{n-1})$ is continuous, open, and surjective by Proposition \ref{pr:CCP}. Since 
\[
\pi_*^{-1}(\Oscr(M,{\bf Q}^{n-2}))=\Oscr(M,{\bf A}_*^{n-1})=\pi^{-1}_*\big(\pi_*(\Oscr(M,{\bf A}_*^{n-1}))\big),
\]
 it follows that the restricted map $\pi_*:  \Oscr(M,{\bf A}_*^{n-1})\to  \Oscr(M,{\bf Q}^{n-2})$ is continuous, open, and surjective as well.
\end{proof}


\section{Completion of the proof of Theorem \ref{th:Gauss-F}}\label{sec:Gauss}
%
%
\noindent
Fix a holomorphic $1$-form $\theta$ on $M$ vanishing nowhere (such exists by \cite{GunningNarasimhan1967MA}, see also \cite[Theorem 1.10.5]{AlarconForstnericLopez2021Book}). The Gauss map assignment $\Gscr:\CMI_{\rm full}(M,\R^n)\to \Oscr_{\rm full}(M,{\bf Q}^{n-2})$ factorizes as $\Gscr=\pi_*\circ \Dscr$, where $\pi_*:\Oscr_{\rm full}(M,{\bf A}_*^{n-1})\to  \Oscr_{\rm full}(M,{\bf Q}^{n-2})$ and $\Dscr: \CMI_{\rm full}(M,\R^n)\to \Oscr_{\rm full}(M,{\bf A}_*^{n-1})$ are the restrictions of the maps in \eqref{eq:pi*} and \eqref{eq:Psi(u)}, respectively. 
We then have the   diagram in \eqref{eq:diagram*}. It trivially follows from Corollary \ref{co:Gauss} that $\pi_*\times {\rm Id}$ is open. To complete the proof of Theorem \ref{th:Gauss-F} we establish the following result.
%
%
\begin{proposition}\label{pr:Psi}
Let $M$ be an open Riemann surface, $\theta$ a holomorphic $1$-form vanishing nowhere on $M$, and $n\ge 3$ an integer.
Then for any $u_0\in  \CMI_{\rm full}(M,\R^n)$ and any neighborhood $U$ of $u_0$ in $\CMI_{\rm full}(M,\R^n)$, there exist neighborhoods $V$ of $\Dscr(u_0)=\partial u_0/\theta$ in $\Oscr(M,{\bf A}_*^{n-1})$  and $W$ of $ \Flux(u_0)$ in $H^1(M,\r^n)$ such that
 $\pi_*(V)\times W\subset (\Gscr,\Flux)(U)$.
\end{proposition}
\begin{proof}
Let $u_0\in U$ be as in the statement.
Let $L\subset M$ be a connected smoothly bounded Runge compact domain and  $\delta>0$ such that
\begin{equation}\label{eq:Udelta}
	U_\delta:=\{u\in  \CMI_{\rm full}(M,\R^n) : \|u-u_0\|_L<\delta\}\subset U.
\end{equation}
For any $\epsilon>0$, consider the neighborhood of $\Dscr(u_0)$ in $\Oscr_{\rm full}(M,{\bf A}_*^{n-1})$ given by 
\[
	V_\epsilon=\{f\in  \Oscr_{\rm full}(M,{\bf A}_*^{n-1}) : \|f-\Dscr(u_0)\|_L<\epsilon\}.
\]
Fix a basis $\Bcal=\{C_1,\ldots,C_l\}$ of the homology group $H_1(L,\z)$, and note that  $L$ being Runge implies that  $\Bcal$ lies in a homology basis of $M$. For any $\sigma>0$, consider the  neighborhood of $\Flux(u_0)$ in $H^1(M,\r^n)$ given by 
\[
	W_\sigma=\{F\in H^1(M,\r^n) : \|F-\Flux(u_0)\|_\Bcal<\sigma\}.
\]
To complete the proof, it suffices to find $\epsilon,\sigma>0$ so small that 
\begin{equation}\label{eq:piV}
\pi_*(V_\epsilon)\times W_\sigma\subset (\Gscr,\Flux)(U_\delta).
\end{equation}
The idea for this is to take the numbers $\epsilon, \sigma>0$ so small that for any $f\in V_\epsilon$ and $F\in W_\sigma$, there is a holomorphic {\em multiplier} $h\colon M\to \c_*$ such that the periods  of $h f \theta$ equal $\imath F$  on $\Bcal$. For this we shall apply the method of period-dominating sprays. In a second step, we shall invoke \cite[Theorem 4.1]{AlarconForstnericLopez2019JGA} (see also \cite[Theorem 5.4.1]{AlarconForstnericLopez2021Book}) to obtain another multiplier $\varphi: M\to \c_*$ so that  $\varphi h f \theta$ has no real periods and its imaginary periods equal $F$ on $\Bcal$. Granting that $h$ and $\varphi$ are close enough to $1$ on $L$, the $1$-form $\varphi h f \theta$ will integrate by the Weierstrass formula in \eqref{eq:conversely} to a conformal minimal immersion in $U_\delta$ with the Gauss map $\pi_*(f)$ and the flux $F$.

Fix $p_0\in \mathring L$. Choose  $\mu>0$  so small that if $f\in V_\mu$ and $\Re(f \theta)$ is exact on $L$ then 
\begin{equation}\label{eq:delta2}
\Big|u_0(p)-u_0(p_0)-\Re \int_{p_0}^p f\theta\Big|<\delta \quad \text{for all $p\in L$.}
\end{equation}
For each $f\in  \Oscr_{\rm full}(M,{\bf A}_*^{n-1})$   define the period map associated to $(\Bcal,f,\theta)$ as the map $\Pcal^f=\big(\Pcal^f_1,\ldots,\Pcal_l^f\big)\colon \Oscr(M)\to \c^n$ given by
\[
\Pcal_j^f(h)=\int_{C_j} hf \theta, \quad h\in \Oscr(M),\quad j=1,\ldots,l,
\]
where, as customary, $\Oscr(M)$ is the space of holomorphic functions on $M$.

By the same ideas in \cite[Lemma 3.2]{AlarconForstnericLopez2019JGA} or \cite[Lemma 5.1.2]{AlarconForstnericLopez2021Book}, there exist finitely many holomorphic functions $g_1,\ldots,g_N$, $N\geq nl$, on $M$ such that the function $\Xi\colon \c^N\times M\to \c_*$ given by
\[
\Xi(\zeta,p)=\prod_{i=1}^Ne^{\zeta_i g_i(p)},\quad \zeta=(\zeta_1,\ldots,\zeta_N)\in \c^N,\quad p\in M,
\]
is a period dominating multiplier of $\Dscr(u_0)$, meaning that the map
\[
\c^N\ni \zeta \longmapsto \Pcal^{\Dscr(u_0)}(\Xi(\zeta,\cdot))\in (\c^n)^l
\]
has maximal rank equal to $ln$ at $\zeta=0$. Since $\Xi(0,\cdot)=1$, there is $\epsilon_1>0$ such that 
\[
\{\Xi(\zeta,\cdot) \Dscr(u_0)\colon \zeta\in \epsilon_1\overline \b_0^N \}  \subset V_{\mu/2}.
\]
Moreover, since $\Xi$ is period dominating  there is $\sigma>0$ such that
\[
\Pcal^{\Dscr(u_0)}(1)+  \sigma\overline{\b}_0^{nl}\subseteq  \{ \Pcal^{\Dscr(u_0)}(\Xi(\zeta,\cdot))\colon \zeta\in \epsilon_1\b_0^N \}.
\]
Since $ \Pcal^f$ depends analytically on $f\in  \Oscr_{\rm full}(M,{\bf A}_*^{n-1})$, there is   $\epsilon>0$  such that the following conditions hold for each $f\in V_{\epsilon}$:
\begin{itemize}
\item The map $\c^N\ni \zeta \mapsto \Pcal^{f}(\Xi(\zeta,\cdot))\in (\c^n)^l$ has rank equal to $ln$ at $\zeta=0$.
\smallskip
\item $\{ \Xi(\zeta,\cdot) f\colon \zeta\in \epsilon_1\overline \b_0^N \}  \subset V_\mu$.
\smallskip
\item $\Pcal^{\Dscr(u_0)}(1)+\sigma\overline{\b}_0^{nl}\subseteq  \{ \Pcal^{f}(\Xi(\zeta,\cdot))\colon \zeta\in \epsilon_1\b_0^N \}$.
\end{itemize}
Let $f\in V_\epsilon$ and $F\in W_\sigma$, and choose  $\zeta_{f,F}\in  \epsilon_1\b_0^N$ such that 
\[
 \Pcal^{f}(\Xi(\zeta_{f,F},\cdot))=  \imath (F(C_1),\ldots,F(C_l));
 \]
 note that $\Re(\Pcal^{\Dscr(u_0)}(1))=0$ since  the real part of $\Dscr(u_0) \theta=\partial u_0$ is exact on $M$.    It turns out that 
$\hat f:=\Xi(\zeta_{f,F},\cdot) f \in V_\mu $ and $\int_{C_j} \hat f \theta=\imath F(C_j)$, $j=1\ldots,l$; in particular, $\Re(\hat f \theta)$ is exact on $L$. 
By \cite[Theorem 4.1]{AlarconForstnericLopez2019JGA} (see also \cite[Theorem 5.4.1]{AlarconForstnericLopez2021Book}) there exists a nowhere vanishing $\varphi\in \Oscr(M)$  such that 
\begin{equation}\label{eq:fluxF}
\int_C \varphi \hat f \theta =\imath F(C)\quad \text{for every loop $C$ in $M$}
\end{equation}
 and  $\varphi$  is so  close to $1$ on $L$ that  $\varphi \hat f \in V_\mu$. 
In particular, $\Re (\varphi \hat f \theta)$ is exact on $M$. Let $\hat u\in  \CMI_{\rm full}(M,\R^n)$ be given by
\[
 \hat u(p)=u_0(p_0)+\Re\int_{p_0}^p \varphi \hat f \theta;
\]
see \eqref{eq:conversely}. By \eqref{eq:delta2} we have that $\|\hat u-u_0\|_L<\delta$, and hence $\hat u\in U_\delta$. Since $\Gscr (\hat u)= \pi_*(\Dscr(\hat u))=\pi_*(\varphi \hat f)=\pi_*(\hat f)=\pi_*(f)$, we have that $\pi_*(f)\in \Gscr (U_\delta)$. Moreover,   \eqref{eq:fluxF} ensures that $\Flux(\hat u)=F$, which implies  \eqref{eq:piV} and completes the proof.
\end{proof}

%
%
\begin{proof}[Proof of Theorem \ref{th:Gauss-F}]
Let $U\subset \CMI_{\rm full}(M,\R^n)$ be an open subset and fix $u\in U$. By Proposition \ref{pr:Psi} there exist neighborhoods $V$ of $\Dscr(u)$ in $\Oscr(M,{\bf A}_*^{n-1})$ and $W$ of $\Flux(u)$ in $H^1(M,\r^n)$ 
 such that $\pi_*(V)\times W\subset (\Gscr,\Flux)(U)$.
By Corollary \ref{co:Gauss}, $\pi_*(V)\subset \Oscr_{\rm full}(M,{\bf Q}^{n-2})$ is open, and hence $(\Gscr,\Flux)(U)$ is a neighborhood of $(\Gscr(u),\Flux(u))$ in $\Oscr_{\rm full}(M,{\bf Q}^{n-2})\times H^1(M,\r^n)$. This shows that $(\Gscr,\Flux)(U)$ is open in $\Oscr_{\rm full}(M,{\bf Q}^{n-2})\times H^1(M,\r^n)$, and so the map in \eqref{eq:G-F} is open. Since it is also continuous and surjective (see Section \ref{ss:TGMA}), this completes the proof.
\end{proof}
%
%
\begin{remark}\label{re:G}
A consequence of Theorem \ref{th:Gauss-F} is that, for any fixed $F\in H^1(M,\R^n)$, the Gauss map assignment $\Gscr:\CMI_{\rm full}^F(M,\R^n)\to\Oscr_{\rm full}(M,{\bf Q}^{n-2})$, restricted to the subspace $\CMI_{\rm full}^F(M,\R^n)\subset \CMI_{\rm full}(M,\R^n)$ of conformal minimal immersions with the flux map $F$, is still an open quotient map. 
In particular, the Gauss map assignment for full holomorphic null curves ${\rm NC}_{\rm full}(M,\C^n)\to\Oscr_{\rm full}(M,{\bf Q}^{n-2})$, sending  each $x\in {\rm NC}_{\rm full}(M,\C^n)$ to $[dx]$, is an open quotient map. (See  \cite[Definition 2.3.3 and Section 3.1]{AlarconForstnericLopez2021Book} for notation and the notion of a null curve.)
\end{remark}


\section{On the Osserman curvature estimate}\label{sec:Gmeagre}

%
%
\subsection{Proof of Theorem \ref{th:Gmeagre}}\label{sec:proofT-2}
 
We obtain Theorem \ref{th:Gmeagre} as an immediate consequence of the following slightly more precise result. 
Recall that a set in a topological space is a $G_\delta$ if it is a countable intersection of open sets. So, a set in a Baire space is residual if and only if it contains a dense $G_\delta$ set.
\begin{theorem}\label{th:residual}
If $M$ is an open Riemann surface, $n\ge 3$ is an integer, and $p_0\in M$, then 
$\{G\in  \Oscr^{\neg K,p_0}_{\rm full}(M,{\bf Q}^{n-2})\colon dG_{p_0}\neq 0\big\}$
is a dense $G_\delta$ set 
in $\Oscr(M,{\bf Q}^{n-2})$.
\end{theorem}
%
%
\begin{proof}  Recall that a point $p\in M$ satisfies $K_u(p)=0$ for some $u\in \CMI(M,\R^n)$ if and only if $p$ is a critical point of $\Gscr(u)$; see \cite[Eq.\ (3.8)]{HoffmanOsserman1980} or, e.g., \cite[Eq.\ (2.87)]{AlarconForstnericLopez2021Book}. 

Define $X=\{u\in\CMI_{\rm full}(M,\R^n): K_u(p_0)\neq 0\}$,
and note that $X$ is open in $\CMI_{\rm full}(M,\R^n)$. Since the Gauss map assignment $\Gscr$ in \eqref{eq:G-F} is surjective by \cite[Theorem 1.1]{AlarconForstnericLopez2019JGA} (see also \cite[Theorem 5.4.1]{AlarconForstnericLopez2021Book}), it turns out that 
\begin{equation}\label{eq:GX}
	\Gscr(X)=\{G\in \Oscr_{\rm full}(M,{\bf Q}^{n-2})\colon dG_{p_0}\neq 0\},
\end{equation}
which is easily seen to be open and dense in $\Oscr(M,{\bf Q}^{n-2})$.
Choose an exhaustion 
\begin{equation}\label{eq:exhaustionL}
L_1\Subset L_2\Subset\cdots \subset \bigcup_{j\ge 1} L_j=M
\end{equation} 
of $M$ by smoothly bounded Runge compact domains such that $p_0\in \mathring L_1$. For each $i,j\in\N$, define
\[
	\Omega_{i,j}=\big\{ u\in X\colon |K_u(p_0)|\,\dist_u(p_0,bL_j)^2>i\big\},
\]
where $\dist_u$ denotes the intrinsic Riemannian distance induced on $M$ by the Euclidean metric in $\R^n$ via $u$. Note that $\Omega_{i,j}$ is open in $X$ for all $i,j\in\N$ by the Cauchy estimates. The set
\[
	\Omega_i=\bigcup_{j\in\N}\Omega_{i,j} =  \big\{ u\in X\colon |K_u(p_0)|\,d_u(p_0)^2>i\big\}
\]
is therefore open as well, and we have in view of \eqref{eq:GX}  that
\begin{multline}\label{eq:bigcap}
	\bigcap_{i\in\N} \Gscr(\Omega_i)=\big\{G\in \Gscr(X): \sup_{u\in \Gscr^{-1}(G)} |K_u(p_0)|\, d_u(p_0)^2=+\infty\big\}
	\\ = \{G\in  \Oscr^{\neg K,p_0}_{\rm full}(M,{\bf Q}^{n-2})\colon dG_{p_0}\neq 0\big\}.
\end{multline}
We claim that 
\begin{equation}\label{eq:GOmegai}
	\text{$\Gscr(\Omega_i)$ is open and dense in $\Oscr(M,{\bf Q}^{n-2})$ for all $i\in\N$}. 
\end{equation}
Indeed, Theorem \ref{th:Gauss-F} guarantees that $\Gscr(\Omega_i)$ is open in $\Oscr_{\rm full}(M,{\bf Q}^{n-2})$, hence in $\Oscr(M,{\bf Q}^{n-2})$, for all $i\in\N$. For the density, since $\Gscr$ is continuous and  $\Gscr(X)$ is dense in $\Oscr(M,{\bf Q}^{n-2})$, it suffices to see that each $\Omega_i$ is dense in $X$. For this, fix $i\in\N$ and choose $u\in X$, a compact set $L\subset M$, and a number $\epsilon>0$. In particular, $K_u(p_0)\neq 0$. We may assume that $p_0\in \mathring L$. By \cite[Theorem 1.2]{AlarconCastro-Infantes2019APDE} (see also \cite[Theorem 3.9.1]{AlarconForstnericLopez2021Book}), there is a complete full conformal minimal immersion $\tilde u:M\to\R^n$ such that $|\tilde u-u|<\epsilon$ on $L$ and $K_{\tilde u}(p_0)=K_u(p_0)\neq 0$ (we actually need not interpolate, it is enough with ensuring that $K_{\tilde u}$ is so close to $K_u$ on $L$ that $K_{\tilde u}(p_0)\neq 0$, which holds by Cauchy estimates provided $\epsilon>0$ is sufficiently small). So, $\tilde u\in X$. By \eqref{eq:exhaustionL} and completeness of $\tilde u$, there exists $j\in\N$ so large that $\dist_{\tilde u}(p_0,bL_j)^2>i/|K_{\tilde u}(p_0)|$, and hence $\tilde u\in \Omega_{i,j}\subset\Omega_i$. This shows that $\Omega_i$ is dense in $X$ for all $i\in\N$, thereby proving \eqref{eq:GOmegai}. Finally, since $\Oscr(M,{\bf Q}^{n-2})$ is a completely metrizable space (see \cite{Arens1946} or, e.g., \cite[p.\ 100]{McCoyNtantu1988}), it follows from \eqref{eq:GOmegai} and the Baire Category Theorem (see, e.g., \cite[Definition 25.1 and Corollary 25.4]{Willard1970}) that
$\bigcap_{i\in\N} \Gscr(\Omega_i)$ is a dense $G_\delta$ set in $\Oscr(M,{\bf Q}^{n-2})$, which completes the proof in view of \eqref{eq:bigcap}.
\end{proof}
%
%
\begin{remark}\label{re:K}
The same proof of Theorem \ref{th:residual} shows the following more general statement. Let  $M$, $n$, and $p_0$ be as in the theorem, and let  $F\in H^1(M,\R^n)$.  Then the set of maps $G\in \Oscr_{\rm full}(M,{\bf Q}^{n-2})$ such that $dG_{p_0}\neq 0$ and there is a sequence of conformal minimal immersions $\{u_j\}_{j\in\N}\subset \Gscr^{-1}(G)$ with $\Flux(u_j)=F$ for all $j\in\N$ and  $\lim_{j\to\infty} |K_{u_j}(p_0)|d_{u_j}(p_0)^2=+\infty$, is a dense $G_\delta$ set in $\Oscr(M,{\bf Q}^{n-2})$. 

In particular, the set of maps $G\in \Oscr_{\rm full}(M,{\bf Q}^{n-2})$ such that $dG_{p_0}\neq 0$ and there is a sequence of full holomorphic null curves $x_j:M\to\C^n$, $j\in\N$, such that $[dx_j]=G$ for all $j\in\N$ and  $\lim_{j\to\infty} |K_{x_j}(p_0)|d_{x_j}(p_0)^2=+\infty$, is a dense $G_\delta$ set in $\Oscr(M,{\bf Q}^{n-2})$.
\end{remark}

%
%
\subsection{Some background in $\r^3$}\label{sec:wei3}
The examples that we shall give in Sections \ref{sec:Example} and \ref{sec:example-2}, concerning the inclusions \eqref{eq:inclu-example} and \eqref{eq:complete}, are minimal surfaces in $\r^3$. Let us recall the Weierstrass formula in this dimension; see e.g. \cite[Sec.\ 2.5 and 2.6]{AlarconForstnericLopez2021Book} and cf.\ \eqref{eq:conversely}.  Let $u=(u_1,u_2,u_3)\colon M\to \r^3$ be a conformal minimal immersion from an open Riemann surface $M$ into $\r^3$ and assume that $u_3$ is not constant. Writing  
\[
\phi_3=2\partial u_3\quad \text{and}\quad g=\frac{\partial u_3}{\partial u_1-\imath \partial u_2},
\]
a holomorphic $1$-form and a meromorphic function on $M$, respectively, we have that
\[
2\partial u=\left(\frac12\big(\frac1{g}-g\big),\frac{\imath}{2}\big(\frac1{g}+g\big),1\right)\phi_3.
\]
Conversely, if $\phi_3\not\equiv 0$ and $g$ are a   holomorphic $1$-form and a meromorphic function on $M$ such that the zeros of $\phi_3$ coincide with the zeros and poles of $g$, with  the same order,  then  
\[
\Phi=(\phi_1,\phi_2,\phi_3)=\left(\frac12\big(\frac1{g}-g\big),\frac{\imath}{2}\big(\frac1{g}+g\big),1\right)\phi_3
\]
is holomorphic and satisfies $\Phi\neq 0$ and $\sum_{j=1}^3 \phi_j^2=0$ everywhere on $M$. If in addition $\Phi$ has no real periods (which always holds if $M$ is simply-connected), then the map  
\[
u\colon M\to \r^3,\quad u(p)=\int^p \Re(\Phi),
\]
is well defined (up to translations) and determines a conformal minimal immersion with $2\partial u=\Phi$. It turns out that $g$ is the complex Gauss map of $u$ and 
\begin{equation}\label{eq:ex-met-min}
2|\partial u|^2= \frac{(1+|g|^2)^2}{2 |g|^2} |\phi_3|^2
\end{equation}
is the metric  on $M$ induced  by the Euclidean one in $\r^3$ via $u$. Furthermore, the Gauss curvature function $K_u\colon M\to (-\infty,0]$ of $u$ is given by 
\begin{equation}\label{eq:forcurv3}
K_u=-\frac{16 |g|^2|dg|^2}{|\phi_3|^2 (1+|g|^2)^4}.
\end{equation}

%
%
\subsection{The inclusion \eqref{eq:inclu-example} is in general proper}\label{sec:Example}
Indeed, let 
\[
	M=\{\zeta\in \c \colon \Im(\zeta)>|\Re(\zeta)|\}
\] 
and consider the holomorphic map $g\colon M\to \cp^1\equiv {\bf Q}^{1}$ given by
\[
	 g(\zeta)=e^\zeta,\quad \zeta \in M.
\]
Note that $g$ has no critical points. 
Every immersion $u\in \Gscr^{-1}(g)\subset {\rm CMI}_{\rm full}(M,\R^3)$ with $4\partial u=\big(\imath(e^\zeta-e^{-\zeta}),e^\zeta+e^\zeta,2\imath\big)d\zeta$  (a piece of a helicoid) satisfies $K_u(\imath t)=-1$ for all $t>0$ and $\lim_{t\to +\infty}  d_u(\imath t)=+\infty$.
This shows that $g\in  \Oscr^{^\neg K}_{\rm full}(M,{\bf Q}^1)$. 

Fix a point $p\in M$ and let us see that $g\in \Oscr^{K,p}_{\rm full}(M,{\bf Q}^1)$. Reason by contradiction and assume that there is a sequence $u_j=(u_{j,1},u_{j,2},u_{j,3})\in\Gscr^{-1}(g)$ such that
\begin{equation}\label{eq:Kuj}
	K_{u_j}(p)=-1 \quad \text{and}\quad d_{u_j}(p)>j\quad \text{for all }j\in\N.
\end{equation}
Write $|\partial u_j|^2=|f_j|^2 (1+|g|^2)^2|d\zeta|^2$ for the  holomorphic function 
\[
f_j=\frac{\partial u_{j,3}}{g}\colon M\to \c_*=\C\setminus\{0\};
\]
 see \eqref{eq:ex-met-min}. Defining $h\colon M\to \c_*$ by $h(\zeta)=e^{-\imath \zeta}$,  we have  that
\[
  (1+|g|^2)^2=\big(1+e^{2 \Re(\zeta)}\big)^2 <\big(1+e^{2 \Im(\zeta)}\big)^2 \leq 4  e^{4 \Im(\zeta)} =4|h|^4,
\]
hence 
\[
2|\partial u_j|^2 \leq \chi_j:=8|f_jh^2|^2|d\zeta|^2.
\]
 Note that $\chi_j$ is a flat metric on $M$ and, by \eqref{eq:Kuj} and the last inequality, the geodesic distance from $p$ to the boundary of $(M,\chi_j)$ satisfies
\begin{equation}\label{eq:dchij}
	d_{\chi_j}(p)>j\quad \text{for all }j\in\N.
\end{equation}
Let $F_j:M\to\C$ be any holomorphic (noncritical) function such that $F_j(p)=0$ and $|F_j'|^2|d\zeta|^2= \chi_j$. In view of \eqref{eq:forcurv3} and \eqref{eq:Kuj}, we have that
\[
	\frac{2|g'|}{|f_j|(1+|g|^2)^2}=1,
\]
and hence
\begin{equation}\label{eq:Fj'}
	|F_j'(p)|=\frac{4 \sqrt{2}|g'h^2|}{(1+|g|^2)^2}(p)\neq 0\quad \text{for all }j\in\N.
\end{equation}

Let us see that for each $j\in\N$ there exist a domain $\Omega_j\subset M$, with $p\in \Omega_j$, and a round disc $D_j\subset \C$, centered at the origin and of  radius larger than ${j}$, such that $F_j:\Omega_j\to D_j$ is a biholomorphism. Indeed, observe that  $F_j\colon (M,\chi_j)\to (\c,|d\zeta|^2)$ is a local isometry. Let $D_j\subset \c$ be any disc centered at the origin with radius $r_j\in (j,d_{\chi_j}(p))$, and denote by $\Omega_j$ the connected component of $F_j^{-1}(D_j)$ containing $p$.  By \eqref{eq:dchij}  any geodesic ray in  $(M,\chi_j)$ starting at $p$ has length greater than $r_j$,  hence $F_j\colon \Omega_j\to D_j$ is a proper surjective local diffeomorphism. Therefore $F_j\colon \Omega_j\to D_j$ is  a covering map and, since $D_j$ is simply-connected, it turns out that $F_j\colon \Omega_j\to D_j$ is a biholomorphism.

The sequence of holomorphic functions $F_j^{-1}:D_j\to M\subset \C$ is normal (note that $\C\setminus M$ is infinite), so there is a limit holomorphic function 
\[
	\varphi=\lim_{j\to\infty}F_j^{-1}:\bigcup_{j\in\N}D_j=\C\to M\cong\D. 
\]
Finally, \eqref{eq:Fj'} implies that  $|(F_j^{-1})'(0)|\neq 0$ does not depend on $j$, and hence $\varphi$ is not constant, a contradiction. This shows that $g\in \Oscr^{K,p}_{\rm full}(M,{\bf Q}^1)$ for all $p\in M$.

%
%
\subsection{The inclusion \eqref{eq:complete} is in general proper}\label{sec:example-2}

An example of this has to be a map without critical points. We shall use the following.
\begin{lemma}\label{le:JX}
 There exists a nowhere vanishing    holomorphic function $g\colon \d\to \c_*$  without critical points  such that
 $(|g|+1/|g|)^2|d\zeta|^2$ is a complete metric on $\d$. 
 
 In particular, there are  complete minimal surfaces  contained in a slab of $\r^3$,  without flat points, and with Gauss map omitting two points of the sphere.
\end{lemma}
\begin{proof}
We adapt the argument in \cite{JorgeXavier1980AM}. Take a  sequence $\{c_j\}_{j\in \n}$ of (connected) circle arcs    in $\d$ centered at the origin and  a sequence of positive real numbers  $\{\epsilon_j\}_{j\in \n}\searrow 0$ such that:
 \begin{enumerate}[{\rm (a)}]
 \item The sequence $\{r_j\}_{j\in \n}\subset ]0,1[$, where $r_j$ is the radius of $c_j$ for all $j\in \n$, is strictly increasing and converging to $1$.
 \item $c_j\cap ]0,1[\neq \varnothing$ for all $j\in \n$ and every divergent path  in $\d$ disjoint from  $ \bigcup_{j\in \n} c_j$ has infinite Euclidean length.
 \item The compact sets $C_j:=\{\zeta \in \c\colon \dist(\zeta,c_j)\leq \epsilon_j\}$, $j\in \n$, are pairwise disjoint  discs  contained in $\d$.
  \end{enumerate}
Consider the closed subset $F= \big(\bigcup_{j\in \n} C_j\big)\cup [0,1[$ in $\d$, and observe that both $F$ and  $\d\setminus F$ are path-connected. Choose a continuous function $f\colon F\to ]0,+\infty[$ such that
 \begin{equation}\label{eq:jo-xa}
 f|_{C_j}=(\epsilon_j)^j \;\;\text{and} \;\;\int_{r_j+\epsilon_j}^{r_{j+1}-\epsilon_{j+1}} f(t)dt>1/\epsilon_{j+1} \text{\;\;for all $j\in \n$}.
 \end{equation}
For any continuous positive function $\delta: F\to \r$,  a  standard recursive application of the classical Runge-Mergelyan  theorem  gives a nowhere vanishing holomorphic function  $h\colon \d\to \c_*$  satisfying $|h-f|<\delta$ on $F$ (recall that $\c_*$ is an Oka manifold). A similar Carleman type argument can be found in  \cite[Theorem 3.8.6]{AlarconForstnericLopez2021Book}.  Consider the nowhere vanishing holomorphic function
 \[
 g:\d\to\c_*,\quad g(\zeta)=e^{\int_0^\zeta h(u)du},
 \]
 and note that $g$ has no critical points since $h$ has no zeroes.
If  the function $\delta$ is suitably chosen then $|g|>e^{1/\epsilon_{j}}$ on $C_{j}$ for every $j\in \n$ sufficiently large. Indeed, observe that $g|_{C_j}$ is approximately constant $g(r_j)$ by the first condition in \eqref{eq:jo-xa} while  $g(r_j)=e^{\int_{0}^{r_{j}} h(u)du} \approx e^{\int_{0}^{r_{j}} f(u)du}>e^{1/\epsilon_j}$  by the second one, whenever that $j$ is large enough and  $\delta$ is sufficiently small.
 In view of {\rm (b)}  and  {\rm (c)}, and reasoning as in \cite{JorgeXavier1980AM},  the metric $(|g|+1/|g|)^2|d\zeta|^2$ is complete.   

For the final assertion of the lemma, consider any conformal minimal immersion  $\d\to \r^3$ with Weierstrass data $(g,\phi_3=d\zeta)$.
\end{proof}

Fix $R>1$ and  call  $D=\{|\zeta|<R\}\subset \c$. Let $g$ be given by  Lemma \ref{le:JX}, and set  
\[
g_0\colon D\to \c_*\subset \CP^1\equiv{\bf Q}^{1}, \quad g_0(\zeta)=g(\zeta^2 R^{-2}).
\]
It is clear that $g_0$ is noncritical on $D_*=D\setminus\{0\}$, 
\begin{equation}\label{eq:g0sim}
  \text{$g_0(-\zeta)=g_0(\zeta)$ for all $\zeta\in D$,}
  \end{equation}
  and
\begin{equation}\label{eq:g0com}
 \text{$(|g_0|+1/|g_0|)^2|d\zeta|^2$ is a complete metric on $D$.}
 \end{equation}

Set  
\[
\text{$M=\{1/R<|\zeta|<R\}\subset D$ \quad and \quad $G=g_0|_M\in \Oscr_{\rm full}(M,{\bf Q}^{1})$.}
\]
We claim that  $G\notin \Oscr^{\rm c}_{\rm full}(M,{\bf Q}^{1})$. Indeed, assume that $ \Gscr^{-1}(G)$ contains a complete  immersion $u \colon M\to \r^3$. Since $g_0( \{|\zeta|\leq 1\})\subseteq \CP^1\equiv \s^2$ has finite spherical area (counting multiplicities) and $G( \{1/R< |\zeta|\leq 1\})\subseteq g_0( \{|\zeta|\leq 1\})$, the complete conformal minimal immersion
$u\colon \{\zeta\in \c\colon 1/R< |\zeta|\leq 1\}\to \r^3$
has finite total curvature (see  \cite[Eq.\ (2.91)]{AlarconForstnericLopez2021Book}). This implies, by Huber's theorem (see \cite{Huber1957} or e.g. \cite[Theorem 2.6.4]{AlarconForstnericLopez2021Book}), that the annular end  $ \{1/R< |\zeta|\leq 1\}$ is parabolic, a contradiction.

To finish, let  us show that $G\in  \Oscr^{^\neg K,\zeta_0}_{\rm full}(M,{\bf Q}^{1})$ for all  $\zeta_0\in M$. 
Indeed, fix $\zeta_0\in M$ and  recall that $G$ is noncritical at $\zeta_0$.  
For each even $j\in \n$ set $f_j\colon D_*\to \c_*$, $f_j(\zeta)=\zeta^{-j}$, and note that \eqref{eq:g0sim} and \eqref{eq:g0com} ensure that
\[
  \text{$f_j (g_0,1/g_0,1)d\zeta$ is exact in $D_*$.}
\]
Thus,  there is an immersion  $u_j\in \Gscr^{-1}(G)\subset  \CMI_{\rm full}(M,\R^3)$ with
\[
2\partial u_j=\frac12\big(G(\zeta)-1/G(\zeta),\imath(G(\zeta)+1/G(\zeta)),2 \big)f_j(\zeta)d\zeta;
\]
see Section \ref{sec:wei3}.
By \eqref{eq:forcurv3} we have that
\begin{equation}\label{eq:g0cur-est}
\text{$ |K_{u_j}(\zeta_0)|= A |\zeta_0|^{2j}$  \quad for \;\;$A=\frac{16|G(\zeta_0)G'(\zeta_0)|^2}{(1+|G(\zeta_0)|^2)^4}\neq 0$.}
    \end{equation}
   %
By \eqref{eq:ex-met-min} we have that  
\[
4|\partial u_j|^2=|f_j|^2 (|G|+1/|G|)^2|d\zeta|^2= (|G|+1/|G|)^2|\zeta|^{-2j}|d\zeta|^2,
\]
which ensures that
 $4|\partial u_j|^2\geq  R^{-2j}(|G|+1/|G|)^2|d\zeta|^2$  on $M$,   proving by \eqref{eq:g0com} that 
 \begin{equation}\label{eq:g0com2}
 \text{$2|\partial u_j|^2$ is complete on $\{1\leq |\zeta|<R\}$.}
\end{equation}
Moreover,  \eqref{eq:ex-met-min},  \eqref{eq:g0com2}, and  the inequality
 \begin{equation}\label{eq:g0com3}
 \text{$2|\partial u_j|^2\geq 2 |\zeta|^{-2j}|d\zeta|^2$
  \quad on \; $M$,}
 \end{equation}
  give  that
\begin{equation}\label{eq:g0du0}
 d_{u_j}(\zeta_0)\geq \sqrt{2}\int_{1/R}^{|\zeta_0|}t^{-j}dt= \frac{ \sqrt{2}}{j-1}(R^{j-1}-|\zeta_0|^{1-j}).
\end{equation}
%
%
By \eqref{eq:g0cur-est}, \eqref{eq:g0du0}, and the fact that $R |\zeta_0|>1$ it turns out that
\[
\lim_{j\to \infty} |K_{u_j}(\zeta_0)|d_{u_j}(\zeta_0)^2\geq \lim_{j\to \infty} \frac{2A |\zeta_0|^2}{(j-1)^2}\big((R |\zeta_0|)^{j-1}-1\big)^2 =+\infty,
\]
proving that  $G\in \Oscr^{^\neg K,\zeta_0}_{\rm full}(M,{\bf Q}^{1})$ as claimed.

%
%
\subsection{Two remarks on homotopy theory}\label{sec:homotopy}
Assume that $\wt M$ is an open Riemann surface and $\pgot \colon \wt M\to M$ is a holomorphic covering map, and consider the map $\pgot_*: \Oscr(M,{\bf Q}^{n-2}) \to \Oscr(\wt M,{\bf Q}^{n-2})$ given by $\pgot_*(G)=G\circ \pgot$ for $G\in \Oscr(M,{\bf Q}^{n-2})$. A first simple observation is that if $\tilde p\in \wt M$ and $p=\pgot(\tilde p)$, then 
\[
	\pgot_*(\Oscr^{^\neg K,p}_{\rm full}(M,{\bf Q}^{n-2}))\subset \Oscr^{^\neg K,\tilde p}_{\rm full}(\wt M,{\bf Q}^{n-2}). 
\]
The proof is an easy exercise and we leave the details to the interested reader.

A second observation is that the spaces $\Oscr^{^\neg K,p}_{\rm full}(M,{\bf Q}^{n-2})$, $p\in M$, and $\Oscr^{^\neg K}_{\rm full}(M,{\bf Q}^{n-2})$ have the same homotopy type as the space $\mathscr{C}(M,{\bf Q}^{n-2})$ of continuous maps $M\to{\bf Q}^{n-2}$ endowed with the compact-open topology. 
Indeed, by \cite[Theorem 1.2]{AlarconLarusson2024Pisa} we have that the inclusions 
\[
	\Oscr^{\rm c}_{\rm full}(M,{\bf Q}^{n-2})\hookrightarrow \Oscr_{\rm full}(M,{\bf Q}^{n-2})\hookrightarrow \mathscr{C}(M,{\bf Q}^{n-2})
\] 
are weak homotopy equivalences (and are homotopy equivalences if $M$ is of finite topological type). Taking into account \eqref{eq:complete} and following the argument in \cite[Proof of Theorem 1.2(a)]{AlarconLarusson2024Pisa}, it turns out that \cite[Theorem 3.1]{AlarconLarusson2024Pisa} implies that the inclusions 
\begin{equation}\label{eq:he}
	\Oscr^{\rm c}_{\rm full}(M,{\bf Q}^{n-2})\hookrightarrow \Oscr^{^\neg K,p}_{\rm full}(M,{\bf Q}^{n-2}) \hookrightarrow
	\Oscr^{^\neg K}_{\rm full}(M,{\bf Q}^{n-2}) \hookrightarrow
	 \Oscr_{\rm full}(M,{\bf Q}^{n-2})
\end{equation}
are weak homotopy equivalences for every $p\in M$ as well (cf.\ \cite[Remark 1.3]{AlarconLarusson2021JGA}). 
%
Therefore, the inclusions in the diagram
\[
\xymatrix{
	\Oscr^{^\neg K,p}_{\rm full}(M,{\bf Q}^{n-2})  \ar@{^{(}->}[d] \ar@{^{(}->}[dr]  
	&   \\
	\Oscr^{^\neg K}_{\rm full}(M,{\bf Q}^{n-2}) \ar@{^{(}->}[r]   &   \mathscr{C}(M,{\bf Q}^{n-2}).
}
\]
are weak homotopy equivalences 
for every point $p\in M$.
Recall that a continuous map $f:X\to Y$ between topological spaces is a {\em weak homotopy equivalence} if it induces a bijection of path components of the two spaces as well as an isomorphism $\pi_k(f):\pi_k(X)\to\pi_k(Y)$ of their homotopy groups for every $k\in \N$ and arbitrary base points. In particular, the three spaces in the above diagram have the same rough topological shape.  This reduces the determination of the homotopy type of $\Oscr^{^\neg K,p}_{\rm full}(M,{\bf Q}^{n-2})$ and $\Oscr^{^\neg K}_{\rm full}(M,{\bf Q}^{n-2})$ to a purely topological problem.


\subsection*{Acknowledgements}
This research was partially supported by the State Research Agency (AEI) via the grants no.\ PID2020-117868GB-I00 and PID2023-150727NB-I00, and the ``Maria de Maeztu'' Unit of Excellence IMAG, reference CEX2020-001105-M, funded by MICIU/AEI/10.13039/501100011033 and ERDF/EU, Spain.




\medskip
\noindent Antonio Alarc\'{o}n, Francisco J. L\'opez
\newline
\noindent Departamento de Geometr\'{\i}a y Topolog\'{\i}a e Instituto de Matem\'aticas (IMAG), Universidad de Granada, Campus de Fuentenueva s/n, E--18071 Granada, Spain.
\newline
\noindent  e-mail: {\tt alarcon@ugr.es}, {\tt fjlopez@ugr.es}

\end{document}